\documentclass{article}
\usepackage{graphicx,verbatim, bussproofs,tikz,float, latexsym}
\usetikzlibrary{arrows}
\usepackage{pgfplots,multicol}
\usepackage{amsmath,amsthm}
\usepackage{amssymb}
\usepackage{stmaryrd}
\usepackage{proof}
\usepackage{cmll}
\usepackage{bussproofs}
\usepackage{authblk}
\DeclareSymbolFont{extraup}{U}{zavm}{m}{n}
\DeclareMathSymbol{\vardiamond}{\mathalpha}{extraup}{87}

\newcommand{\commment}[1]{}

\newcommand{\nomi}{\mathbf{i}}
\newcommand{\nomj}{\mathbf{j}}

\newcommand{\bigamp}{\mathop{\mbox{\Large \&}}}

\renewcommand{\phi}{\varphi}
\renewcommand{\emptyset}{\varnothing}

\newcommand{\Diamondblack}{\vardiamond}

\renewcommand{\epsilon}{\varepsilon}
\theoremstyle{definition}
\newtheorem{theorem}{Theorem}[section]
\newtheorem{lemma}[theorem]{Lemma}

\newtheorem{proposition}[theorem]{Proposition}

\newtheorem{example}[theorem]{Example}

\newtheorem{corollary}[theorem]{Corollary}

\newtheorem{definition}[theorem]{Definition}
\newtheorem{remark}[theorem]{Remark}

\title{Correspondence Theory for Modal Fairtlough-Mendler Semantics of Intuitionistic Modal Logic}
\author{Zhiguang Zhao}

\date{}

\begin{document}
\maketitle

\begin{abstract}
We study the correspondence theory of intuitionistic modal logic in modal Fairtlough-Mendler semantics (modal FM semantics) \cite{FaMe97}, which is the intuitionistic modal version of possibility semantics \cite{Ho16}. We identify the fragment of inductive formulas \cite{GorankoV06} in this language and give the algorithm $\mathsf{ALBA}$ \cite{CoPa12} in this semantic setting. There are two major features in the paper: one is that in the expanded modal language, the nominal variables, which are interpreted as atoms in perfect Boolean algebras, complete join-prime elements in perfect distributive lattices and complete join-irreducible elements in perfect lattices, are interpreted as the refined regular open closures of singletons in the present setting, similar to the possibility semantics for classical normal modal logic \cite{Zh21d}; the other feature is that we do not use conominals or diamond, which restricts the fragment of inductive formulas significantly. We prove the soundness of the $\mathsf{ALBA}$ with respect to modal FM frames and show that the $\mathsf{ALBA}$ succeeds on inductive formulas, similar to existing settings like \cite{CoPa12,Zh21d,Zh22a}.

\end{abstract}

\section{Introduction}\label{Sec:Intro}

\paragraph{Possibility Semantics.} Possibility semantics was proposed by Humberstone \cite{Hu81}, which is a generalization of possible world semantics for modal logic, based on partial possibilities instead of complete possible worlds like the one in standard possible world semantics. In recent years there have been a lot of studies in possibility semantics \cite{BeHo20,DiHo20,HT16b,HT16a,Ho14,vBBeHo16}, to name but a few. For a comprehensive study of possibility semantics, see \cite{Ho21,Ho16}.

\paragraph{Intuitionistic Study of Possibility Semantics.} In \cite{BeHo16}, Bezhanishvili and Holliday use the tools of nuclei to study the equivalence between Fairtlough-Mendler semantics (FM semantics for short) \cite{FaMe97}, Dragalin semantics \cite{Dr79,Dr88} and nuclear semantics of intuitionistic logic, which can be regarded as different ways to realize possibility semantics of intuitionistic logic. In \cite{BeHo19}, Bezhanishvili and Holliday study the different semantics of intuitionistic logic, which form a hierarchy. Among these semantics, Dragalin semantics is more like neighbourhood semantics of modal logic \cite{Mo70,Pa17,Sc70} and Beth semantics of intuitionistic logic \cite{Be56}, nuclear semanrics is more like algebraic semantics, and FM semantics is more like relational semantics with two relations. In \cite{Ma16}, Massas provides choice-free representation theorems for distributive lattice, Heyting algebras and co-Heyting algebras. In \cite{Ma22}, Massas studies the B-frame duality, where B-frames can be seen as a generalization of posets, which play an important role in the representation theory of Heyting algebras.

\paragraph{Correspondence Theory for Possibility Semantics.} In \cite{Ya16}, Yamamoto studies Sahlqvist correspondence theory for full possibility frames. In \cite[Theorem 7.20]{Ho16}, Holliday shows that all inductive formulas are filter-canonical, therefore every normal modal logic axiomatized by inductive formulas is sound and complete with respect to its canonical full possibility frame. In \cite{Zh21d}, Zhao shows that inductive formulas have first-order correspondents in full possibility frames as well as in filter-descriptive possibility frames, using algebraic and algorithmic correspondence theory methods \cite{CoGhPa14,CoPa12}. In this setting, the algebraic strcuture of regular open subsets of a given full possibility frame is a complete Boolean algebra with a complete operator, which is not necessarily an atomic Boolean algebra, therefore it is not necessarily perfect.

\paragraph{Our Methodology.} Our aim is to see if Sahlqvist-type correspondence theory could work for semantics whose algebraic counterpart is based on locales, i.e.\ complete Heyting algebras, which are not necessarily perfect (where every element is join-generated by complete join-primes).

We study the correspondence theory of intuitionistic modal logic in the modal version of Fairtlough-Mendler semantics, using algorithmic correspondence theory methods, as explained in \cite{CoGhPa14,CoPa12}. We define the class of inductive formulas for this semantics as well as the Ackermann Lemma Based Algorithm $\mathsf{ALBA}$ which computes the first-order correspondents of inductive formulas. 

Following the methodology of \cite{Zh21d}, our semantic analysis of the modal FM frames is base on their dual algebraic structures, which are complete Heyting algebras with complete operators (not necessarily perfect), where complete join-primes are not always available, in contrast to settings like \cite{CoPa12}. We use the representation of complete Heyting algebras as the refined regular open subsets in Fairtlough-Mendler frames, and we identify one Heyting algebra with an operator (HAO) and one Boolean algebra with an operator (BAO) as the dual algebraic structures of a given modal FM frame $\mathbb{F}$: namely, the HAO $\mathbb{H}_{\mathsf{RO}_{12}}$ of refined regular open subsets of $\mathbb{F}$ and the BAO $\mathbb{B}_{\mathsf{K}}$ of arbitrary subsets of $\mathbb{F}$. Therefore, we can define a natural order-embedding $e:\mathbb{H}_{\mathsf{RO}_{12}}\to\mathbb{B}_{\mathsf{K}}$. It is completely meet-preserving, therefore it has a left adjoint $c:\mathbb{B}_{\mathsf{K}}\to\mathbb{H}_{\mathsf{RO}_{12}}$ sending a subset of the domain of $\mathbb{F}$ to the smallest refined regular open subset containing it. We will use $c$ substantially in the interpretations of the expanded modal language $\mathcal{L}^{+}$ of the algorithm $\mathsf{ALBA}$, which will play an important role in the relational semantics and the refined regular open translation of the expanded modal language.

To summarize, the principles for choosing the interpretations of the expanded modal language are the following:

\begin{itemize}
\item The set of possible interpretations of nominals is join-dense in the complete algebra of possible values of propositional variables;
\item The set of possible interpretations of conominals is meet-dense in the complete algebra of possible values of propositional variables (only in settings when this is possible);
\item The interpretation of black connectives should be the adjoints of modalities in the basic propositional language;
\item These interpretations should be expressible in a first-order way.
\end{itemize}

\paragraph{Structure of the paper.} 
The paper is organized as follows: Section \ref{Sec:Prelim} introduces some relevant structures and notions that will be used in later sections. Section \ref{aSec:Prelim} presents preliminaries on modal Fairtlough-Mendler semantics. Section \ref{Sec:Semantic:Environment} analyzes the semantic environment for the interpretation of the expanded modal language. Section \ref{aSec:expanded:language} introduces the expanded modal language formally as well as the refined regular open translation. Section \ref{asec:Sahlqvist} gives the syntactic definition of inductive formulas. Section \ref{aSec:ALBA} introduced the algorithm $\mathsf{ALBA}$ with an example. Section \ref{aSec:success} gives its success proof on inductive formulas and Section \ref{aSec:soundness} gives its soundness proof. Section \ref{aSec:Discussion} gives the conclusions.

\section{Preliminaries}\label{Sec:Prelim}

In this section, we introduce some topological and relational structures and the notion of nucleus, which will be useful in later sections. For more details, see \cite{BeHo16,Ma16,Ma18}.

\subsection{Topological and Relational Structures}

\begin{definition}[Refined Bitopological Space, Definition 5.1.1 in \cite{Ma16}]
A \emph{refined bitopological space} is a triple $(X,\tau_1,\tau_2)$ where $\tau_1$ and $\tau_2$ are topologies on $X$ and $\tau_1\subseteq\tau_2$.
\end{definition}

\begin{definition}[Refined Birelational Frame]
A \emph{refined birelational frame} is a tuple $(X,\leq_1,\leq_2)$ where $\leq_2\subseteq\leq_1$ are both partial orders on $X$.
\end{definition}

For any poset $(X,\leq)$, define the topology $\tau_{\leq}$ taking all the upsets of $(X,\leq)$ as open sets. Therefore we can identify a partial order $(X,\leq)$ with its corresponding Alexandroff topology $(X,\tau)$. Since $\leq_2\subseteq\leq_1$ iff $\tau_{\leq_1}\subseteq\tau_{\leq_2}$, we can identify refined birelational frames with their corresponding refined bi-Alexandroff spaces.

In what follows, we will call refined birelational frames also Fairtlough-Mendler frames (FM frames for short) \cite{BeHo16,FaMe97}.

\subsection{Nucleus}
In this subsection we introduce the notion of nucleus on Heyting algebras. For more details, see \cite{Jo82,Ma76,Ma81}.

\begin{definition}[Nucleus]
A \emph{nucleus} on a Heyting algebra $\mathbb{H}$ is a map $j:\mathbb{H}\to \mathbb{H}$ satisfying the following conditions:

\begin{itemize}
\item $j(a\wedge b)=j(a)\wedge j(b)$;
\item $a\leq j(a)$. 
\end{itemize}

It is clear that a nucleus on $\mathbb{H}$ is also a closure operator, and the following properties follow from the definition:

\begin{itemize}
\item $j(\top)=\top$;
\item if $a\leq b$, then $j(a)\leq j(b)$;
\item $j(j(a))=j(a)$.
\end{itemize}
We call a nucleus \emph{dense} if $j(\bot)=\bot$.
\end{definition}

\begin{definition}[Nuclear algebra]
A \emph{nuclear algebra} is a pair $(\mathbb{H},j)$ where $\mathbb{H}$ is a Heyting algebra and $j$
is a nucleus on $\mathbb{H}$. It is a \emph{localic nuclear algebra} if $\mathbb{H}$ is a locale, i.e.\ a complete Heyting algebra.
\end{definition}

\begin{theorem}
If $\mathbb{H}=(H,\bot,\top,\land,\lor,\to,j)$ is a nuclear algebra, then $\mathbb{H}_{j}=(H_{j},\bot_{j},\top,\land_{j},\lor_{j},\to_{j})$ is a Heyting algebra where $H_{j}=\{a\in H\mid a=j(a)\}$ and for $a,b\in H_{j}$:

\begin{itemize}
\item $\bot_{j}=j(\bot)$;
\item $a\land_{j}b=a\land b$;
\item $a\lor_{j}b=j(a\lor b)$;
\item $a\to_{j}b=a\to b$.
\end{itemize}
If $H$ is a localic nuclear algebra, then $H_j$ is a locale, where for $Y\subseteq H_{j}$:

\begin{itemize}
\item $\bigwedge_{j}Y=\bigwedge Y$;
\item $\bigvee_{j}Y=j(\bigvee Y)$.
\end{itemize}
\end{theorem}

In a given refined bitopological space $(X,\tau_1,\tau_2)$, we use $\mathsf{I}_{i}$ to denote the interior map in $\tau_i$ and $\mathsf{C}_{i}$ to denote the closure map in $\tau_i$. We use $\mathsf{RO}_{12}(X)$ to denote $\{Y\subseteq X\mid Y=\mathsf{I}_1\mathsf{C}_2(Y)\}$.

\begin{lemma}
Let $\mathbb{F}=(X,\leq_1,\leq_2)$ be an FM-frame and $(X,\tau_1,\tau_2)$ be its corresponding refined bitopological space, let the $\mathbb{H}_{\mathsf{RO}_{12}}$ be the algebra $(\mathsf{RO}_{12}(X)$, $\emptyset$, $H$, $\cap$, $\lor_{\mathsf{RO}_{12}}$, $\to_{\mathsf{RO}_{12}})$ where $Y\lor_{\mathsf{RO}_{12}}Z:=\mathsf{I}_1\mathsf{C}_2(Y\cup Z)$, $Y\to_{\mathsf{RO}_{12}}Z:=\mathsf{I}_1((X-Y)\cup Z)$. Then the operator $\mathsf{I}_1\mathsf{C}_2$ is a nucleus on the complete Heyting algebra of opens in $\tau_1$. Therefore $\mathbb{H}_{\mathsf{RO}_{12}}$ is a complete Heyting algebra.
\end{lemma}

\section{Preliminaries on Modal Fairtlough-Mendler semantics}\label{aSec:Prelim}

In the present section we collect the preliminaries on modal Fairtlough-Mendler semantics. For more details, see e.g.,\ \cite{FaMe97,Ma16,Ma18}. The style of presentation follows \cite{CoPa12,Zh21d}.

\subsection{Language}

Let $\mathsf{Prop}$ be the set of propositional variables. The basic modal language $\mathcal{L}$ is defined as follows:
$$\phi::=p\mid\bot\mid\top\mid\phi\land\phi\mid\phi\lor\phi\mid\phi\to\phi\mid\Box\phi,$$ where $p\in\mathsf{Prop}$.

In the algorithm, we will use \emph{inequalities} $\phi\leq\psi$, whose truth in a model is equivalent to the global truth of $\phi\to\psi$ in the model. We will also define \emph{quasi-inequalities} $\phi_1\leq\psi_1\ \&\ \ldots\ \&\ \phi_n\leq\psi_n\Rightarrow\phi\leq\psi$, where $\&$ is the meta-level conjunction and $\Rightarrow$ is the meta-level implication. We say that a formula is \emph{pure} if it does not contain occurrences of propositional variables.

\subsection{Semantics}

For any $R\subseteq W\times W$, we denote $R[X]=\{w\in W\mid (\exists x\in X)Rxw\}$, $R^{-1}[X]=\{w\in W\mid (\exists x\in X)Rwx\}$, $R[w]:=R[\{w\}]$ and $R^{-1}[w]:=R^{-1}[\{w\}]$, respectively.

\begin{definition}[Modal FM frames and models]

A \emph{modal FM frame} is a tuple $\mathbb{F}=(X,\leq_1,\leq_2,R)$, where $(X,\leq_1,\leq_2)$ is an FM frame, $R\subseteq W\times W$ such that $\Box_{\mathsf{RO}_{12}}(X):=\{w\in W\mid R[w]\subseteq X\}\in\mathsf{RO}_{12}(X)$ for any $X\in\mathsf{RO}_{12}(X)$. A \emph{modal FM model} is a pair $\mathbb{M}=(\mathbb{F},V)$ where $V:\mathsf{Prop}\to\mathsf{RO}_{12}(X)$ is a \emph{valuation} on $\mathbb{F}$.
\end{definition}

Given any modal FM model $\mathbb{M}=(X,\leq_1,\leq_2,R,V)$ and any $w\in X$, the \emph{satisfaction relation} is defined as follows:

\begin{center}
\begin{tabular}{l c l}
$\mathbb{M}, w\Vdash p$ & iff & $w\in V(p)$\\
$\mathbb{M}, w\Vdash\bot$ & : & never\\
$\mathbb{M}, w\Vdash\top$ & : & always\\
$\mathbb{M}, w\Vdash \phi\land\psi$ & iff & $\mathbb{M}, w\Vdash \phi$ and $\mathbb{M}, w\Vdash\psi$\\
$\mathbb{M}, w\Vdash \phi\lor\psi$ & iff & $(\forall v\geq_1 w)(\exists u\geq_2 v)(\mathbb{M}, u\Vdash \phi$ or $\mathbb{M}, u\Vdash\psi)$\\
$\mathbb{M}, w\Vdash \phi\to\psi$ & iff & $(\forall v\geq_1 w)(\mathbb{M}, v\Vdash\phi\ \Rightarrow\ \mathbb{M}, v\Vdash\psi)$\\
$\mathbb{M}, w\Vdash \Box\phi$ & iff & $\forall v(Rwv\ \Rightarrow\ \mathbb{M}, v\Vdash\phi)$.\\
\end{tabular}
\end{center}

We use $\llbracket\phi\rrbracket^{\mathbb{M}}=\{w\in X\mid \mathbb{M}, w\Vdash\phi\}$ to denote the \emph{truth set} of $\phi$ in $\mathbb{M}$. We say that $\phi$ is \emph{globally true} on $\mathbb{M}$ (notation: $\mathbb{M}\Vdash\phi$) if $\mathbb{M},w\Vdash\phi$ for all $w\in X$. We say that $\phi$ is \emph{valid} on $\mathbb{F}$ (notation: $\mathbb{F}\Vdash\phi$) if $\phi$ is globally true on $(\mathbb{F},V)$ for all $V$ on $X$. The semantics for inequalities and quasi-inequalities is as follows:
\begin{center}
\begin{tabular}{l c l}
$\mathbb{M}\Vdash \phi\leq\psi$ & iff & (for all $w\in X$)($\mathbb{M}, w\Vdash\phi\ \Rightarrow\ \mathbb{M}, w\Vdash\psi$)\\
$\mathbb{M}\Vdash\bigamp_{i=1}^{n}(\phi_i\leq\psi_i)\Rightarrow\phi\leq\psi$ & iff & ($\mathbb{M}\Vdash\phi_i\leq\psi_i$ for all $i)\ \Rightarrow\ (\mathbb{M}\Vdash\phi\leq\psi$).\\
\end{tabular}
\end{center}
An inequality (resp.\ quasi-inequality) is valid on $\mathbb{F}$ if it is globally true on $(\mathbb{F}, V)$ for all $V$.

\begin{proposition}\label{aprop:inequalities}
For any modal FM model $\mathbb{M}=(X,\leq_1,\leq_2,R,V)$ and any $w\in X$,
\begin{center}
\begin{tabular}{l c l}
$\mathbb{M}\Vdash \phi\to\psi$ & iff & $\llbracket\phi\rrbracket^{\mathbb{M}}\subseteq\llbracket\psi\rrbracket^{\mathbb{M}}$ iff $\mathbb{M}\Vdash \phi\leq\psi$\\
$\mathbb{F}\Vdash\phi\to\psi$ & iff & $\mathbb{F}\Vdash\phi\leq\psi$.\\
\end{tabular}
\end{center}
\end{proposition}

\section{Semantic Environment of our Setting}\label{Sec:Semantic:Environment}

In this section, we use some algebraic structures to give the semantic definition how the expanded modal language used in the algorithm will be defined. The style of presentation follows \cite{Zh21d}.

\subsection{The Heyting Algebra with Operator $\mathbb{H}_{\mathsf{RO}_{12}}$} \label{aSec:Algebraic:Semantics}

\begin{definition}[Heyting algebra with operator]
A \emph{Heyting algebra with operator} (HAO) is a tuple $\mathbb{H}=(H,\bot,\top,\land,\lor,\to,\Box)$, where the $\Box$-free part is a Heyting algebra and $\Box$ is a unary operation such that $\Box\top=\top$ and $\Box(a\land b)=\Box a\land \Box b$ for any $a,b\in H$. An HAO is \emph{complete} if its Heyting algebra part is complete. An HAO is \emph{completely multiplicative} if $\Box$ preserves all existing meets.
\end{definition}

\begin{definition}\label{Def:dual:BAO}
For any modal FM frame $\mathbb{F}=(X,\leq_1,\leq_2,R)$, let the $\mathbb{H}_{\mathsf{RO}_{12}}=(\mathsf{RO}_{12}(X),\emptyset,H,\cap,\lor_{\mathsf{RO}_{12}},\to_{\mathsf{RO}_{12}},\Box_{\mathsf{RO}_{12}})$ where $Y\lor_{\mathsf{RO}_{12}}Z:=\mathsf{I}_1\mathsf{C}_2(Y\cup Z)$, $Y\to_{\mathsf{RO}_{12}}Z:=\mathsf{I}_1((X-Y)\cup Z)$.
\end{definition}

\begin{proposition}(cf.\ \cite[Theorem 5.6(2)]{Ho16})
For any modal FM frame $\mathbb{F}$, $\mathbb{H}_{\mathsf{RO}_{12}}$ is a complete and completely multiplicative HAO.
\end{proposition}

The essential difference between the correspondence for Kripke semantics for intuitionistic modal logic and the current setting is that the algebra $\mathbb{H}_{\mathsf{RO}_{12}}$ is not perfect in general, i.e.\ they are not join-generated by completely join-primes.

\subsection{The Auxiliary BAO $\mathbb{B}_{\mathsf{K}}$}

Consider any modal FM-frame $\mathbb{F}=(X,\leq_1,\leq_2,R)$, there is another way of viewing it, namely taking it as a trirelational frame $\mathbb{F}_{3}=(X,\leq_1,\leq_2,R)$, the complex algebra of which is a Boolean algebra with three operators $\mathbb{B}_{\mathsf{K}}$.

The formal definition of the BAO $\mathbb{B}_{\mathsf{K}}$ is given as follows:

\begin{definition}
For any modal FM-frame $\mathbb{F}=(X,\leq_1,\leq_2,R)$, define $\mathbb{B}_{\mathsf{K}}=(P(X),\emptyset,W,\cap, \cup,-,\Box_{\mathsf{K}},\Box_{\leq_1},\Box_{\leq_2})$, where $\cap, \cup, -$ are set-theoretic intersection, union and complementation respectively, $\Box_{\mathsf{K}}(a)=\{w\in W\mid R[w]\subseteq a\}$, and $\Box_{\leq_i}(a)=\{w\in W\mid (\forall v\geq_i w)(v\in a)\}$.
\end{definition}

Since $\mathsf{RO}_{12}(X)\subseteq P(X)$, we have that $e:\mathbb{H}_{\mathsf{RO}_{12}}\hookrightarrow\mathbb{B}_{\mathsf{K}}$ is an order-embedding according to the order of the two algebras. Since in $\mathbb{H}_{\mathsf{RO}_{12}}$, arbitrary intersections of refined regular open subsets are again refined regular open, $e$ is completely meet-preserving. Notice also that $\Box_{\mathsf{RO}_{12}}$ is the restriction of $\Box_{\mathsf{K}}$ to $\mathbb{H}_{\mathsf{RO}_{12}}$. All these observations can be summarized as follows:

\begin{lemma}\label{alemma:preserve:e}
$e:\mathbb{H}_{\mathsf{RO}_{12}}\hookrightarrow\mathbb{B}_{\mathsf{K}}$ is a completely meet-preserving order-embedding such that $e\circ \Box_{\mathsf{RO}_{12}}=\Box_{\mathsf{K}}\circ e$.
\end{lemma}

However, since $\mathbb{H}_{\mathsf{RO}_{12}}$ and $\mathbb{B}_{\mathsf{K}}$ have different definitions of join, $\mathbb{H}_{\mathsf{RO}_{12}}$ is \emph{not} a sublattice of $\mathbb{B}_{\mathsf{K}}$. 

We have the following corollary by \cite[Proposition 7.34]{DaPr90}):

\begin{corollary}\label{acor:existence:left:adjoint}
$e:\mathbb{H}_{\mathsf{RO}_{12}}\hookrightarrow\mathbb{B}_{\mathsf{K}}$ has a left adjoint $c:\mathbb{B}_{\mathsf{K}}\to\mathbb{H}_{\mathsf{RO}_{12}}$ defined, for every $a\in \mathbb{B}_{\mathsf{K}}$, by
\begin{center}
$c(a)=\bigwedge_{\mathbb{H}_{\mathsf{RO}_{12}}}\{b\in\mathbb{H}_{\mathsf{RO}_{12}}\mid a\leq e(b)\}.$
\end{center}
\end{corollary}

Clearly, $c(Y)=\mathsf{I}_1\mathsf{C}_2(Y)$ for any $\leq_1$-upset $Y\subseteq X$, and $c(Y)=\mathsf{I}_1\mathsf{C}_2(\uparrow_1\hspace{-1mm}Y)$ for any $Y\subseteq W$ where $\uparrow_1\hspace{-1mm}Y$ is the least $\leq_1$-upset containing $Y$. Indeed, by definition $c(X)=\bigwedge_{\mathbb{H}_{\mathsf{RO}_{12}}}\{Y\in\mathsf{RO}_{12}(X)\mid X\leq e(Y)\}=\bigcap\{Y\in\mathsf{RO}_{12}(X)\mid X\subseteq Y\}$, which is the smallest refined regular open set containing $X$. The closure operator $c$ will be called the \emph{refined regular open closure map} and $c(a)$ the \emph{refined regular open closure} of $a$. 

\subsection{Interpreting the Additional Symbols}

\subsubsection{Nominals}

Nominals are originally introduced in hybrid logic (see e.g.\ \cite[Chapter 14]{BeBlWo06}). They are used as special propositional variables and are interpreted as singletons, to refer to a world of the domain. In correspondence theory, nominals are used for computing the minimal valuation of propositional variables so as to eliminate them. 

In the literature, nominals are interpreted as atoms (in complete atomic Boolean algebra based settings), completely join-prime elements (in perfect distributive lattice based settings), completely join-irreducible elements (in perfect (non-distributive) lattice based settings), regular open closures of singletons (in the setting of possibility semantics of classical normal modal logic). The common feature is that the selected class of elements join-generates the relevant complete lattices. 

Therefore, in our setting, we need to find a subset of $\mathbb{H}_{\mathsf{RO}_{12}}$ which join-generates the whole $\mathbb{H}_{\mathsf{RO}_{12}}$. We take $\mathsf{Nom}(\mathbb{H}_{\mathsf{RO}_{12}}):=\{c(\{x\})\mid x\in X\}$ be the set of refined regular open closures of singletons in $\mathbb{B}_{\mathsf{K}}$, which will be shown to be the join-dense set of $\mathbb{H}_{\mathsf{RO}_{12}}$.

\begin{center}
\begin{tabular}{|c|c|}

\hline
\textbf{Propositional base}&\textbf{Nominals/join-generators}\\
\hline
perfect Boolean algebras&atoms\\
\hline
perfect distributive lattices&complete join-primes\\
\hline
perfect general lattices&complete join-irreducibles\\
\hline
possibility semantics&regular open closures of singletons\\
\hline
\end{tabular}
\end{center}

\begin{proposition}\label{aprop:regular:open:join}
For any $Y\in\mathbb{H}_{\mathsf{RO}_{12}}$, $$Y=\bigvee_{\mathsf{RO}_{12}}\{c(\{x\})\mid x\in Y\}=\bigvee_{\mathsf{RO}_{12}}\{Z\in\mathsf{Nom}(\mathbb{H}_{\mathsf{RO}_{12}})\mid Z\subseteq Y\}.$$
\end{proposition}
\begin{proof}
For the first equality, it is easy to see that for any $x\in Y$, $c(\{x\})\subseteq c(Y)=Y$, so $\bigvee_{\mathsf{RO}_{12}}\{c(\{x\})\mid x\in Y\}\subseteq Y$. For the other direction, for any $x\in Y$, $x\in c(\{x\})\subseteq\bigvee_{\mathsf{RO}_{12}}\{c(\{x\})\mid x\in Y\}$. The second equality is easy.
\end{proof}

\subsubsection{Black Diamond}

The black diamond $\Diamondblack$ comes from tense logic, which is the backward looking diamond ``there was a moment where\ldots''. When $\Box$ is interpreted on the relation $R$, $\Diamondblack$ is interpreted on the relation $R^{-1}$. These two connectives has the following property: $\phi\to\Box\psi$ is globally true in a model iff $\Diamondblack\phi\to\psi$ is globally true in that model. This property is called the \emph{adjunction property} algebraically. In correspondence theory, we use the black diamond to compute the minimal valuation of a propositional variable, and the adjunction property is exactly what we need. Therefore, our task here is to find the left adjoint of $\Box_{\mathsf{RO}_{12}}$.

We know from lattice theory that in complete lattices, a map has a left adjoint iff it is completely meet-preserving. Since $\mathbb{H}_{\mathsf{RO}_{12}}$ and $\mathbb{B}_{\mathsf{K}}$ are both complete, and $\Box_{\mathsf{RO}_{12}}:\mathbb{H}_{\mathsf{RO}_{12}}\to\mathbb{H}_{\mathsf{RO}_{12}}$ and $\Box_{\mathsf{K}}:\mathbb{B}_{\mathsf{K}}\to\mathbb{B}_{\mathsf{K}}$ are both completely meet-preserving, therefore they both have left adjoints, which are denoted by $\Diamondblack_{\mathsf{RO}_{12}}$ and $\Diamondblack_{\mathsf{K}}$, respectively. 

\begin{lemma}(Folklore.)\label{alem:Diamond:Full}
$\Diamondblack_{\mathsf{K}}(X)=R[X]$ for any $X\subseteq W$.
\end{lemma}

\begin{lemma}\label{alem:Diamond:RO}
$\Diamondblack_{\mathsf{RO}_{12}}(Y)=(c\circ\Diamondblack_{\mathsf{K}}\circ e)(Y)$.
\end{lemma}
\begin{proof}
We have the following chain of equalities:
\begin{center}
\begin{tabular}{r c l l}
$\Diamondblack_{\mathsf{RO}_{12}}(Y)$ & = & $\bigwedge_{\mathsf{RO}_{12}}\{Z\in\mathbb{H}_{\mathsf{RO}_{12}}\mid Y\leq\Box_{\mathsf{RO}_{12}}(Z)\}$ &\\
& = &$\bigwedge_{\mathsf{RO}_{12}}\{Z\in\mathbb{H}_{\mathsf{RO}_{12}}\mid e(Y)\leq(\Box_{\mathsf{K}}\circ e)(Z)\}$ & (Lemma \ref{alemma:preserve:e})\\
& = &$\bigwedge_{\mathsf{RO}_{12}}\{Z\in\mathbb{H}_{\mathsf{RO}_{12}}\mid (\Diamondblack_{\mathsf{K}}\circ e)(Y)\leq e(Z)\}$ &\\
& = &$\bigwedge_{\mathsf{RO}_{12}}\{Z\in\mathbb{H}_{\mathsf{RO}_{12}}\mid (c\circ\Diamondblack_{\mathsf{K}}\circ e)(Y)\leq Z\}$ &\\
& = &$(c\circ\Diamondblack_{\mathsf{K}}\circ e)(Y)$\\
& = &$(c\circ\Diamondblack_{\mathsf{K}})(Y)$ &\\
& = &$c(R[Y])$ & (Lemma \ref{alem:Diamond:Full})\\
& = &$\mathsf{I}_1\mathsf{C}_2(\uparrow_1\hspace{-1mm}R[Y])$. & (Corollary \ref{acor:existence:left:adjoint})
\end{tabular}
\end{center}
\end{proof}

\section{Preliminaries on Algorithmic Correspondence}\label{aSec:expanded:language}

In the present section, we collect preliminaries on algorithmic correspondence for modal FM semantics. We will define the expanded modal language $\mathcal{L}^{+}$ for $\mathsf{ALBA}$, the first-order correspondence language $\mathcal{L}^{1}$ and the refined regular open translation of $\mathcal{L}^{+}$ into $\mathcal{L}^{1}$. Our treatment is similar to \cite{CoPa12,Zh21d}.

\subsection{The Expanded Modal Language $\mathcal{L}^{+}$}\label{asec:expanded:modal:language}

The expanded modal language $\mathcal{L}^{+}$ contains the basic modal language $\mathcal{L}$ and a set $\mathsf{Nom}$ of \emph{nominals}, to be interpreted as elements in $\mathsf{Nom}(\mathbb{H}_{\mathsf{RO}_{12}})$, and the \emph{black connective} $\Diamondblack$, to be interpreted as the left adjoint of $\Box$. Its definition is given as follows:
$$\phi::=p\mid\bot\mid\top\mid\nomi\mid\phi\land\phi\mid\phi\lor\phi\mid\phi\to\phi\mid\Box\phi\mid\Diamondblack\phi,$$
where $p\in\mathsf{Prop}$ and $\nomi\in\mathsf{Nom}$. 

We extend the valuation $V$ to nominals such that $V(\nomi)\in\mathsf{Nom}(\mathbb{H}_{\mathsf{RO}_{12}})$ and use $i$ to denote the element that $V(\nomi)=\mathsf{I}_1\mathsf{C}_2(\uparrow_1\hspace{-1mm}\{i\})$.

The satisfaction relation for the additional symbols is given as follows:
\begin{definition}
In any modal FM model $\mathbb{M}=(X,\leq_1,\leq_2,R,V)$,
\begin{center}
\begin{tabular}{l c l}
$\mathbb{M}, w\Vdash\nomi$ & iff & $(\forall v\geq_1 w)(\exists u\geq_2 v)(i\leq_1 u)$.\\
$\mathbb{M}, w\Vdash \Diamondblack\phi$ & iff & $(\forall v\geq_1 w)(\exists u\geq_2 v)(\exists t\leq_1 u)\exists s(Rst$ and $\mathbb{M}, s\Vdash \phi)$.\\
\end{tabular}
\end{center}
\end{definition}

Truth set and validity are defined similarly to the basic modal language.

\subsection{The Correspondence Languages $\mathcal{L}^1$}\label{asec:correspondence:language}

In order to express the first-order correspondents of modal formulas, we need to define the first-order correspondence language $\mathcal{L}^1$.

\paragraph{Syntax.}
The non-logical symbols of the first-order correspondence language $\mathcal{L}^{1}$ consists of the following:
\begin{itemize}
\item a set of unary predicate symbols $P_n$, each of which corresponds to a propositional variable $p_n$ and is going to be interpreted as a refined regular open subset of the domain;
\item three binary relation symbols $\leq_1,\leq_2$ and $R$ corresponding to the relations denoted with the same symbol;
\item a set of individual symbols $i_n$, each of which corresponds to a nominal $\nomi_n$. Notice that we allow the individual symbols $i_n$ to be quantified by quantifiers $\forall i_n,\exists i_n$.
\end{itemize}

\subsection{The Refined Regular Open Translation}\label{asec:regular:open:translation}

In the present section, we will define the \emph{refined regular open translation} of $\mathcal{L}^{+}$ into $\mathcal{L}^{1}$.

\begin{definition}[Syntactic Refined Regular Open Closure]
Given a first-order formula $\alpha(x)$ with at most $x$ free, the \emph{syntactic regular open closure} $\mathsf{RO}^{12}_{x}(\alpha(x))$ is defined as $(\forall y\geq_1 x)(\exists z\geq_2 y)(\exists z'\leq_1 z)\alpha(z')$.
\end{definition}

\begin{definition}[Refined Regular Open Translation]
The regular open translation of a formula in $\mathcal{L}^{+}$ into $\mathcal{L}^{1}$ is given as follows:

\begin{center}
\begin{tabular}{r c l}
$ST_x(p_{i})$ & := & $P_{i}x$;\\
$ST_x(\bot)$ & := & $x\neq x$;\\
$ST_x(\top)$ & := & $x=x$;\\
$ST_x(\nomi)$ & := & $\mathsf{RO}^{12}_{x}(i=x)$;\\
$ST_x(\phi_1\land\phi_2)$ & := & $ST_x(\phi_1)\land ST_x(\phi_2)$;\\
$ST_x(\phi_1\lor\phi_2)$ & := & $(\forall y\geq_1 x)(\exists z\geq_2 y)(ST_z(\phi_1)\lor ST_z(\phi_2))$;\\
$ST_x(\phi_1\to\phi_2)$ & := & $(\forall y\geq_1 x)(ST_y(\phi_1)\to ST_y(\phi_2))$;\\
$ST_x(\Box\phi)$ & := & $\forall y(Rxy\to ST_y(\phi))$;\\
$ST_x(\Diamondblack\phi)$ & := & $\mathsf{RO}^{12}_{x}(\exists y(Ryx\land ST_y(\phi)))$.\\
\end{tabular}
\end{center}
\end{definition}

The refined regular open translations of inequalities $\phi\leq\psi$ and quasi-inequalities are given as follows (notice that they are interpreted on the level of global truth of models):\\

$ST(\phi\leq\psi):=\forall x(ST_x(\phi)\to ST_x(\psi))$;

$ST(\bigamp_{j=1}^{n}(\phi_j\leq\psi_j)\Rightarrow\phi\leq\psi):=\bigwedge_{j=1}^{n}\forall x(ST_x(\phi_j)\to ST_x(\psi_j))\to\forall x(ST_x(\phi)\to ST_x(\psi)).$\\

The following proposition justifies the translation defined above:

\begin{proposition}\label{aprop:translation}
For any modal FM model $\mathbb{M}=(X,\leq_1,\leq_2,R,V)$, any $w\in X$, any $\mathcal{L}^{+}$-formula $\phi$, any $\mathcal{L}^{+}$-inequality $\mathsf{Ineq}$, any $\mathcal{L}^{+}$-quasi-inequality $\mathsf{Quasi}$,
\begin{itemize}
\item $\mathbb{M},w\Vdash\phi\mbox{ iff }\mathbb{M}\models ST_x(\phi)[w];$
\item $\mathbb{M}\Vdash\phi\mbox{ iff }\mathbb{M}\models \forall xST_x(\phi);$
\item $\mathbb{M}\Vdash\mathsf{Ineq}\mbox{ iff }\mathbb{M}\models ST(\mathsf{Ineq});$
\item $\mathbb{M}\Vdash\mathsf{Quasi}\mbox{ iff }\mathbb{M}\models ST(\mathsf{Quasi}).$
\end{itemize}
\end{proposition}

\section{Inductive formulas}\label{asec:Sahlqvist}

In this section, we define inductive formulas for our setting. The definition is similar to \cite{Zh22a}.

We define \emph{positive formulas} with variables in $A\subseteq\mathsf{Prop}$ as follows:
$$\mathsf{POS}_{A}::=p\mid\bot\mid\top\mid\Box\mathsf{POS}_{A}\mid\mathsf{POS}_{A}\land\mathsf{POS}_{A}\mid\mathsf{POS}_{A}\lor\mathsf{POS}_{A}$$
where $p\in A$.

We define the \emph{dependence order} on propositional variables as any irreflexive and transitive binary relation $<_\Omega$ on them. 

We define the \emph{PIA formulas}\footnote{For the name, see e.g.\ \cite[Remark 3.24]{PaSoZh16}.} with main variable $p$ as follows:
$$\mathsf{PIA}_{p}::=p\mid\bot\mid\top\mid\Box\mathsf{PIA}_{p}\mid\mathsf{POS}_{A_{p}}\to\mathsf{PIA}_{p}$$

where $A_{p}=\{q\in \mathsf{Prop}\mid q<_{\Omega}p\}$. 

We define the \emph{inductive antecedent} as follows:
$$\mathsf{Ant}::=\mathsf{PIA}_{p}\mid\mathsf{Ant}\land\mathsf{Ant}\mid\mathsf{Ant}\lor\mathsf{Ant}$$

where $p\in\mathsf{Prop}$. 

We define the \emph{inductive succedent} as follows:
$$\mathsf{Suc}::=\mathsf{POS}_{\mathsf{Prop}}\mid\mathsf{PIA}_{q}\to\mathsf{Suc}\mid\Box\mathsf{Suc}\mid\mathsf{Suc}\land\mathsf{Suc}$$

where $q\in\mathsf{Prop}$.

Finally, an \emph{$\Omega$-inductive formula} is a formula of the form $\mathsf{Ant}\to\mathsf{Suc}$. An \emph{inductive formula} is an $\Omega$-inductive formula for some $<_{\Omega}$.

\begin{remark}
Since in our settings, we have not yet found meet-dense set of $\mathbb{H}_{\mathsf{RO}_{12}}$, and we will only compute minimal valuations rather than maximal valuations, we do not use the order-type and signed generation tree style definitions like in \cite{CoPa12}. In addition, since we do not have diamond in the basic language, the fragment that we have is much smaller than typical definitions of inductive formulas in some existing settings like \cite{CoPa12}.
\end{remark}

\section{The Algorithm $\mathsf{ALBA}$}\label{aSec:ALBA}

In the present section, we define the algorithm $\mathsf{ALBA}$ which compute the first-order correspondence of the input formula, in the style of \cite{CoPa12,Zh22a}. 

The algorithm $\mathsf{ALBA}$ executes in three stages. $\mathsf{ALBA}$ receives a formula $\mathsf{Ant}\to\mathsf{Suc}$ as input and transforms it into the inequality $\mathsf{Ant}\leq\mathsf{Suc}$.

\begin{enumerate}

\item \textbf{Preprocessing and First approximation}:

\begin{enumerate}
\item We apply the following \emph{distribution rules} exhaustively:

\begin{itemize}
\item In $\mathsf{Ant}$, rewrite every subformula of the former form into the latter form:

\begin{itemize}
\item $(\beta\lor\gamma)\land\alpha$, $(\beta\land\alpha)\lor(\gamma\land\alpha)$
\item $\alpha\land(\beta\lor\gamma)$, $(\alpha\land\beta)\lor(\alpha\land\gamma)$
\end{itemize}
\item In $\mathsf{Suc}$, rewrite every subformula of the former form into the latter form:
\begin{itemize}
\item $(\beta\land\gamma)\lor\alpha$, $(\beta\lor\alpha)\land(\gamma\lor\alpha)$
\item $\alpha\lor(\beta\land\gamma)$, $(\alpha\lor\beta)\land(\alpha\lor\gamma)$
\item $\alpha\to\beta\land\gamma$, $(\alpha\to\beta)\land(\alpha\to\gamma)$
\item $\Box(\alpha\land\beta)$, $\Box\alpha\land\Box\beta$
\end{itemize}
\end{itemize}

\item Apply the \emph{splitting rules}:

$$\infer{\alpha\leq\beta\ \ \ \alpha\leq\gamma}{\alpha\leq\beta\land\gamma}
\qquad
\infer{\alpha\leq\gamma\ \ \ \beta\leq\gamma}{\alpha\lor\beta\leq\gamma}
$$

\end{enumerate}

Now for each obtained inequality $\phi_i\leq\psi_i$, We apply the \emph{first-approximation rule}:
$$\infer{\nomi_0\leq\phi_i\ \Rightarrow\ \nomi_0\leq\psi_i}{\phi_i\leq\psi_i}$$

Now we call each quasi-inequality $\nomi_0\leq\phi_i\ \Rightarrow\ \nomi_0\leq\psi_i$ a \emph{system}, and use $\mathsf{S}$ to denote a meta-conjunction of inequalities. When $\mathsf{S}$ is empty, we denote it as $\emptyset$.

\item \textbf{The reduction-elimination cycle}:

In this stage, for each system $\nomi_0\leq\phi_i\ \Rightarrow\ \nomi_0\leq\psi_i$, we apply the following rules to eliminate all the propositional variables:
\begin{enumerate}
\item \emph{Splitting rules}:
$$\infer{(\mathsf{S}\ \Rightarrow\ \alpha\leq\beta)\ \ \ (\mathsf{S}\ \Rightarrow\ \alpha\leq\gamma)}{\mathsf{S}\ \Rightarrow\ \alpha\leq\beta\land\gamma}$$
$$\infer{\mathsf{S}\ \&\ \alpha\leq\beta\ \&\ \alpha\leq\gamma\ \Rightarrow\ \phi\leq\psi}{\mathsf{S}\ \&\ \alpha\leq\beta\land\gamma\ \Rightarrow\ \phi\leq\psi}$$
$$\infer{(\mathsf{S}\ \Rightarrow\ \alpha\leq\gamma)\ \ \ (\mathsf{S}\ \Rightarrow\ \beta\leq\gamma)}{\mathsf{S}\ \Rightarrow\ \alpha\lor\beta\leq\gamma}$$
$$\infer{\mathsf{S}\ \&\ \alpha\leq\gamma\ \&\ \beta\leq\gamma\ \Rightarrow\ \phi\leq\psi}{\mathsf{S}\ \&\ \alpha\lor\beta\leq\gamma\ \Rightarrow\ \phi\leq\psi}$$

\item \emph{Residuation rules}:\label{bPage:Residuation:Rules}
$$\infer{\mathsf{S}\ \Rightarrow\ \Diamondblack\alpha\leq\beta}{\mathsf{S}\ \Rightarrow\ \alpha\leq\Box\beta}$$
$$\infer{\mathsf{S}\ \&\ \Diamondblack\alpha\leq\beta\ \Rightarrow\ \phi\leq\psi}{\mathsf{S}\ \&\ \alpha\leq\Box\beta\ \Rightarrow\ \phi\leq\psi}$$
$$\infer{\mathsf{S}\ \Rightarrow\ \alpha\land\beta\leq\gamma}{\mathsf{S}\ \Rightarrow\ \alpha\leq\beta\to\gamma}$$
$$\infer{\mathsf{S}\ \&\ \alpha\land\beta\leq\gamma\ \Rightarrow\ \phi\leq\psi}{\mathsf{S}\ \&\ \alpha\leq\beta\to\gamma\ \Rightarrow\ \phi\leq\psi}$$

\item \emph{Approximation rule}:
$$\infer{\mathsf{S}\ \&\ \nomi\leq\phi\ \Rightarrow\ \nomi\leq\psi}{\mathsf{S}\ \Rightarrow\ \phi\leq\psi}$$

The nominal introduced by the approximation rule must not occur in the system before applying the rule.

\item \emph{Deleting rules}:

$$\infer{\mathsf{S}\ \Rightarrow\ \phi\leq\psi}{\mathsf{S}\ \&\ \alpha\leq\top\ \Rightarrow\ \phi\leq\psi}$$
$$\infer{\emptyset\ \Rightarrow\ \alpha\leq\top}{\mathsf{S}\ \Rightarrow\ \alpha\leq\top}$$

\item \emph{Right-handed Ackermann rule}. This rule eliminates propositional variables and is the core of the algorithm:
$$\infer{\bigamp_{j=1}^{m}\eta_j(\theta/p)\leq\iota_j(\theta/p)\ \Rightarrow\ \phi(\theta/p)\leq\psi(\theta/p)}{\bigamp_{i=1}^{n}\theta_i\leq p\ \&\ \bigamp_{j=1}^{m}\eta_j\leq\iota_j\ \Rightarrow\ \phi\leq\psi}$$

where:
\begin{enumerate}
\item $p$ does not occur in $\theta_1, \ldots, \theta_n$;
\item Each $\eta_i,\psi$ is positive, and each $\iota_i,\phi$ negative in $p$, for $1\leq i\leq m$;
\item $\theta:=\theta_1\lor\ldots\lor\theta_n$.
\end{enumerate}
\end{enumerate}

\item \textbf{Output}: If in Stage 2, the algorithm gets stuck for some systems, i.e.\ some propositional variables cannot be eliminated, then the algorithm stops and output ``failure''. Otherwise, each initial system after the first approximation rule has been reduced to a set of pure quasi-inequalities Reduce$(\nomi_0\leq\phi_i\ \Rightarrow\ \nomi_0\leq\psi_i)$, and then the output is a set of pure quasi-inequalities $\bigcup_{i\in I}$Reduce$(\nomi_0\leq\phi_i\ \Rightarrow\ \nomi_0\leq\psi_i)$. Then we can use the conjunction of the refined regular open translations of the quasi-inequalities to obtain the first-order correspondence (notice that in the standard translation of each quasi-inequality, we need to universally quantify over all the individual variables).
\end{enumerate}

\begin{example}
Here we give an example. For the sake of clarity we add propositional quantifiers and nominal quantifiers before the quasi-inequality.

$\ $\\
$\forall p(\Box p\to p)$\\
$\forall p(\Box p\leq p)$\\
$\forall p\forall\nomi(\nomi\leq\Box p\ \Rightarrow \nomi\leq p)$\\
$\forall p\forall\nomi(\Diamondblack\nomi\leq p\ \Rightarrow \nomi\leq p)$\\
$\forall\nomi(\nomi\leq\Diamondblack\nomi)$\\
$\forall i\forall x(ST_{x}(\nomi)\to ST_{x}(\Diamondblack\nomi))$\\
$\forall i\forall x(\mathsf{RO}^{12}_{x}(i=x)\to \mathsf{RO}^{12}_{x}(\exists y(Ryx\land \mathsf{RO}^{12}_{x}(i=y))))$
\end{example}

\section{Success}\label{aSec:success}

In the present section, we show the success of $\mathsf{ALBA}$ on any inductive formula $\phi\to\psi$. 

\begin{theorem}
$\mathsf{ALBA}$ succeeds on any inductive formula $\phi\to\psi$ and outputs a set of pure quasi-inequalities and a first-order formula.
\end{theorem}

\begin{proof}
We check the shape of the inequalitys or systems in each stage, for the input formula $\mathsf{Ant}\to\mathsf{Suc}$:

\textbf{Stage 1.} 

After applying the distribution rules, it is easy to see that $\mathsf{Ant}$ becomes the form $\bigvee\bigwedge\mathsf{PIA}_{p}$, and $\mathsf{Suc}$ becomes of the form $\bigwedge\mathsf{Suc}'$, where $$\mathsf{Suc}'::=\mathsf{POS}_{\mathsf{Prop}}\mid\mathsf{PIA}_{q}\to\mathsf{Suc}'\mid\Box\mathsf{Suc}'.$$

Then by applying the splitting rules, we get a set of inequalities of the form $\bigwedge\mathsf{PIA}_{p}\leq\mathsf{Suc}'$.

After the first approximation rule, we get systems of the form $\nomi_0\leq\bigwedge\mathsf{PIA}_{p}\ \Rightarrow\ \nomi_0\leq\mathsf{Suc}'$.

\textbf{Stage 2.}
In this stage, we deal with each system $\nomi_0\leq\bigwedge\mathsf{PIA}_{p}\ \Rightarrow\ \nomi_0\leq\mathsf{Suc}'$.

For the inequality $\nomi_0\leq\bigwedge\mathsf{PIA}_{p}$, by first applying the splitting rule for $\land$ and then exhaustively applying the residuation rules for $\Box$ and $\to$, we get inequalities of the form $\mathsf{MinVal}_{p}\leq p$ or $\mathsf{MinVal}_{p}\leq \top$ or $\mathsf{MinVal}_{p}\leq \bot$, where 
$$\mathsf{MinVal}_{p}::=\nomi_0\mid\Diamondblack\mathsf{MinVal}_{p}\mid \mathsf{MinVal}_{p}\land\mathsf{POS}_{A_{p}},$$
where $A_{p}=\{q\in \mathsf{Prop}\mid q<_{\Omega}p\}$.

Now we deal with the $\nomi_0\leq\mathsf{Suc}'$ part. 

\begin{itemize}
\item If the system is of the form $\mathsf{S}\ \&\ \bigamp(\mathsf{MinVal}_{p}\leq p)\ \Rightarrow\ \nomi_0\leq\mathsf{PIA}_{q}\to\mathsf{Suc}'$, then we have the following execution of $\mathsf{ALBA}$:

$\mathsf{S}\ \&\ \bigamp(\mathsf{MinVal}_{p}\leq p)\ \Rightarrow\ \nomi_0\land\mathsf{PIA}_{q}\leq\mathsf{Suc}'$\\
$\mathsf{S}\ \&\ \bigamp(\mathsf{MinVal}_{p}\leq p)\ \&\ \nomj\leq\nomi_0\land\mathsf{PIA}_{q}\ \Rightarrow\ \nomj\leq\mathsf{Suc}'$\\
$\mathsf{S}\ \&\ \bigamp(\mathsf{MinVal}_{p}\leq p)\ \&\ \nomj\leq\nomi_0\ \&\ \nomj\leq\mathsf{PIA}_{q}\ \Rightarrow\ \nomj\leq\mathsf{Suc}'$\\
$\mathsf{S}\ \&\ \bigamp(\mathsf{MinVal}_{p}\leq p)\ \&\ \nomj\leq\nomi_0\ \Rightarrow\ \nomj\leq\mathsf{Suc}'.$

\item If the system is of the form $\bigamp(\mathsf{MinVal}_{p}\leq p)\ \Rightarrow\ \nomi_0\leq\Box\mathsf{Suc}'$, then we have the following execution of $\mathsf{ALBA}$:

$\bigamp(\mathsf{MinVal}_{p}\leq p)\ \Rightarrow\ \Diamondblack\nomi_0\leq\mathsf{Suc}'$\\
$\bigamp(\mathsf{MinVal}_{p}\leq p)\ \&\ \nomj\leq\Diamondblack\nomi_0\ \Rightarrow\ \nomj\leq\mathsf{Suc}'$
\end{itemize}
Therefore, by the reduction strategies above, we get a quasi-inequality of the form (here $\mathsf{Pure}$ is a meta-conjunction of inequalities without propositional variables):
$$\bigamp(\mathsf{MinVal}_{p}\leq p)\ \&\ \mathsf{Pure}\ \Rightarrow\ \nomj\leq \mathsf{POS}_{\mathsf{Prop}}.$$

Now we can apply the right-handed Ackermann rule to an $\Omega$-miminal variable $q$ to eliminate it. Since there are only finitely many propositional variables, we can always find another $\Omega$-miminal variable to eliminate. Finally we eliminate all propositional variables and get pure quasi-inequalities and the refined regular open translation.
\end{proof}

\section{Soundness}\label{aSec:soundness}

In this section, we prove the soundness of $\mathsf{ALBA}$ with respect to modal FM frames. The basic proof structure is similar to \cite{CoPa12,CP:constructive,Zh22a}.

\begin{theorem}[Soundness]\label{bThm:Soundness}
If $\mathsf{ALBA}$ runs according to the success proof in the previous section on an input inductive formula $\phi\to\psi$ and outputs a first-order formula $\mathsf{FO(\phi\to\psi)}$, then for any modal FM frame $\mathbb{F}=(X,\leq_1,\leq_2,R)$, $$\mathbb{F}\Vdash\mathsf{\phi\to\psi}\mbox{ iff }\mathbb{F}\vDash\mathsf{FO(\phi\to\psi)}.$$
\end{theorem}

\begin{proof}
The proof goes similarly to \cite[Theorem 8.1]{CoPa12}. Let $\phi\to\psi$ denote the input formula, $\{\nomi_0\leq\phi_i\ \Rightarrow\ \nomi_0\leq\psi_i\}_{i\in I}$ denote the set of systems after the first-approximation rule, let $\{\mbox{Reduce}(\nomi_0\leq\phi_i\ \Rightarrow\ \nomi_0\leq\psi_i)\}_{i\in I}$ denote the sets of quasi-inequalities after Stage 2, let $\mathsf{FO(\phi\to\psi)}$ denote the refined regular open translation of the quasi-inequalities in Stage 3 into first-order formulas, then it suffices to show the equivalence from (\ref{bCrct:Eqn0}) to (\ref{bCrct:Eqn4}) given below:

\begin{eqnarray}
&&\mathbb{F}\Vdash\phi\to\psi\label{bCrct:Eqn0}\\
&&\mathbb{F}\Vdash\nomi_0\leq\phi_i\ \Rightarrow\ \nomi_0\leq\psi_i,\mbox{ for all }i\in I\label{bCrct:Eqn2}\\
&&\mathbb{F}\Vdash\mbox{Reduce}(\nomi_0\leq\phi_i\ \Rightarrow\ \nomi_0\leq\psi_i),\mbox{ for all }i\in I\label{bCrct:Eqn3}\\
&&\mathbb{F}\vDash\mathsf{FO(\phi\to\psi)}\label{bCrct:Eqn4}
\end{eqnarray}

The equivalence between (\ref{bCrct:Eqn0}) and (\ref{bCrct:Eqn2}) follows from Proposition \ref{prop:Soundness:stage:1};

The equivalence between (\ref{bCrct:Eqn2}) and (\ref{bCrct:Eqn3}) follows from Propositions \ref{bProp:Stage:2}, \ref{bProp:Ackermann};

The equivalence between (\ref{bCrct:Eqn3}) and (\ref{bCrct:Eqn4}) follows from Proposition \ref{aprop:translation}.
\end{proof}

In what follows, we will prove the soundness of each rule in each stage.

\begin{proposition}[Soundness of the rules in Stage 1]\label{prop:Soundness:stage:1}
The distribution rules and the splitting rules are sound in both directions in $\mathbb{F}$.
\end{proposition}

\begin{proof}
For the soundness of the distribution rules, it follows from the validity of the following equivalences in $\mathbb{F}$:
\begin{itemize}
\item $(\alpha\lor\beta)\land\gamma\leftrightarrow(\alpha\land\gamma)\lor(\beta\land\gamma)$;
\item $\alpha\land(\beta\lor\gamma)\leftrightarrow(\alpha\land\beta)\lor(\alpha\land\gamma)$;
\item $\Box(\alpha\land\beta)\leftrightarrow\Box\alpha\land\Box\beta$;
\item $(\alpha\land\beta)\lor\gamma\leftrightarrow(\alpha\lor\gamma)\land(\beta\lor\gamma)$;
\item $\alpha\lor(\beta\land\gamma)\leftrightarrow(\alpha\lor\beta)\land(\alpha\lor\gamma)$;
\item $(\alpha\to\beta\land\gamma)\leftrightarrow(\alpha\to\beta)\land(\alpha\to\gamma)$.
\end{itemize}

For the soundness of the splitting rules, it follows from the following fact:

$$\mathbb{F}\vDash\alpha\leq\beta\land\gamma\mbox{ iff }(\mathbb{F}\vDash\alpha\leq\beta\mbox{ and }\mathbb{F}\vDash\alpha\leq\gamma);$$
$$\mathbb{F}\vDash\alpha\lor\beta\leq\gamma\mbox{ iff }(\mathbb{F}\vDash\alpha\leq\gamma\mbox{ and }\mathbb{F}\vDash\beta\leq\gamma).$$

For the soundness of the first-approximation rule, by Proposition \ref{aprop:regular:open:join}, $$V(\phi_i)=\bigvee_{\mathsf{RO}_{12}}\{Z\in\mathsf{Nom}(\mathbb{H}_{\mathsf{RO}_{12}})\mid Z\subseteq V(\phi_i)\}.$$
Now we have the following chain of equivalences:
\begin{center}
\begin{tabular}{r l}
& $\mathbb{F}\Vdash\phi_i\leq\psi_i$\\
iff & for any $V$, $\bigvee_{\mathsf{RO}_{12}}\{Z\in\mathsf{Nom}(\mathbb{H}_{\mathsf{RO}_{12}})\mid Z\subseteq V(\phi_i)\}\subseteq V(\psi_i)$\\
iff & for any $V$, $(\forall Z\in \mathsf{Nom}(\mathbb{H}_{\mathsf{RO}_{12}})\mbox{ s.t. }Z\subseteq V(\phi_i))(Z\subseteq V(\psi_i))$\\
iff & for any $V$, $(\forall Z\in \mathsf{Nom}(\mathbb{H}_{\mathsf{RO}_{12}})(Z\subseteq V(\phi_i)\ \Rightarrow\ Z\subseteq V(\psi_i))$\\
iff & for any $V'$, $V'(\nomi)\subseteq V'(\phi_i)\ \Rightarrow\ V'(\nomi)\subseteq V'(\psi_i)$\\
iff & $\mathbb{F}\Vdash\nomi\leq\phi_i\ \Rightarrow\ \nomi\leq\psi_i$.\\
\end{tabular}
\end{center}
\end{proof}

For the soundness of each rule in Stage 2, we introduce the following notations: for each rule, before its application we have a system $\mathsf{S}\ \Rightarrow \mathsf{Ineq}$, after its application we get a system $\mathsf{S}'\ \Rightarrow \mathsf{Ineq}'$ (indeed, for the splitting rules we might have two systems after the application, but their soundness are trivial), the soundness of Stage 2 is then the equivalence of the following:
\begin{itemize}
\item $\mathbb{F}\Vdash\mathsf{S}\ \Rightarrow \mathsf{Ineq}$
\item $\mathbb{F}\Vdash\mathsf{S}'\ \Rightarrow \mathsf{Ineq}'$
\end{itemize}

It suffices to show the following property for the splitting rules, residuation rules and the deleting rules:
\begin{center}
For any $\mathbb{F}$, any $V$, 

$\mathbb{F},V\Vdash\mathsf{S}\ \Rightarrow \mathsf{Ineq}$ iff $\mathbb{F},V\Vdash\mathsf{S}'\ \Rightarrow \mathsf{Ineq}'.$
\end{center}

For the first-approximation rule, we prove it directly.

For the right-handed Ackermann rule, we also prove it directly.

\begin{proposition}\label{bProp:Stage:2}
The splitting rules, the approximation rule, the residuation rules and the deleting rules in Stage 2 are sound in both directions in $\mathbb{F}$.
\end{proposition}

\begin{proof}
The soundness proofs for the splitting rules and the residuation rules are similar to the soundness of the same rules in \cite{CoPa12}.
\begin{itemize}
\item For the splitting rules, it follows from the following equivalence: for any modal FM frame $\mathbb{F}$, any valuation $V$ on $\mathbb{F}$, 
$$\mathbb{F},V\Vdash\alpha\leq\beta\land\gamma\mbox{ iff }\mathbb{F},V\Vdash\alpha\leq\beta\mbox{ and }\mathbb{F},V\Vdash\alpha\leq\gamma;$$
$$\mathbb{F},V\Vdash\alpha\lor\beta\leq\gamma\mbox{ iff }\mathbb{F},V\Vdash\alpha\leq\gamma\mbox{ and }\mathbb{F},V\Vdash\beta\leq\gamma.$$

\item For the residuation rules, it follows from the following equivalence: for any modal FM frame $\mathbb{F}$, any valuation $V$ on $\mathbb{F}$, 

\begin{itemize}
\item $\mathbb{F},V\Vdash\alpha\leq\Box\beta\mbox{ iff }\mathbb{F},V\Vdash\Diamondblack\alpha\leq\beta;$
\item $\mathbb{F},V\Vdash\alpha\leq\beta\to\gamma\mbox{ iff }\mathbb{F},V\Vdash\alpha\land\beta\leq\gamma.$
\end{itemize}

The equivalences above follow from the fact that the interpretations of $\Diamondblack$ and $\Box$ form an adjunction pair, and the interpretations of $\land$ and $\to$ form a residuation pair.

\item For the approximation rule, the soundness proof is similar to the first-approximation rule: for any modal FM frame $\mathbb{F}$, any valuation $V$, 

\begin{center}
\begin{tabular}{r l}
& $\mathbb{F},V\Vdash\phi_i\leq\psi_i$\\
iff & $V(\phi_i)\subseteq V(\psi_i)$\\
iff & $\bigvee_{\mathsf{RO}_{12}}\{Z\in\mathsf{Nom}(\mathbb{H}_{\mathsf{RO}_{12}})\mid Z\subseteq V(\phi_i)\}\subseteq V(\psi_i)$\\
iff & $(\forall Z\in \mathsf{Nom}(\mathbb{H}_{\mathsf{RO}_{12}})\mbox{ s.t. }Z\subseteq V(\phi_i))(Z\subseteq V(\psi_i))$\\
iff & $(\forall Z\in \mathsf{Nom}(\mathbb{H}_{\mathsf{RO}_{12}})(Z\subseteq V(\phi_i)\ \Rightarrow\ Z\subseteq V(\psi_i))$.\\
\end{tabular}
\end{center}

Therefore, if $\mathbb{F}\Vdash\mathsf{S}\ \Rightarrow\ \phi_i\leq\psi_i$, then for any $V$, if $\mathbb{F},V\Vdash\mathsf{S}$ and $V(\nomi)\subseteq V(\phi_i)$, then $V(\phi_i)\subseteq V(\psi_i)$, so for $V(\nomi)\in \mathsf{Nom}(\mathbb{H}_{\mathsf{RO}_{12}})$, from the above equivalences we have $V(\nomi)\subseteq V(\psi_i)$, i.e. $\mathbb{F},V\Vdash\nomi\leq\psi_i$. Thus $\mathbb{F}\Vdash\mathsf{S}\ \&\ \nomi\leq\phi_i\ \Rightarrow\ \nomi\leq\psi_i$.

If $\mathbb{F}\Vdash\mathsf{S}\ \&\ \nomi\leq\phi_i\ \Rightarrow\ \nomi\leq\psi_i$, then for any valuation $V$, if $\mathbb{F},V\Vdash\mathsf{S}$, then for any $Z\in\mathsf{Nom}(\mathbb{H}_{\mathsf{RO}_{12}})$, if $Z\subseteq V(\phi_i)$, take $V':=V^{\nomi}_{Z}$, then since $\nomi$ does not occur in $\mathsf{S}$, we have $\mathbb{F},V'\Vdash\mathsf{S}$. We also have $V'(\nomi)=Z\subseteq V(\phi_i)=V'(\phi_i)$, so from $\mathbb{F}\Vdash\mathsf{S}\ \&\ \nomi\leq\phi_i\ \Rightarrow\ \nomi\leq\psi_i$ we get $Z=V'(\nomi)\subseteq V'(\psi_i)=V(\psi_i)$, therefore $\mathbb{F},V\Vdash\phi_i\leq\psi_i$. Therefore we get $\mathbb{F}\Vdash\mathsf{S}\ \Rightarrow\ \phi_i\leq\psi_i$.

\item The soundness of the deleting rule is trivial, since $\alpha\leq\top$ always holds in any modal FM model.
\end{itemize}
\end{proof}

\begin{proposition}\label{bProp:Ackermann}
The right-handed Ackermann rule is sound in $\mathbb{F}$.
\end{proposition}

\begin{proof}
Without loss of generality we assume that $n=1$ and $m=1$. Then it suffices to show the following equivalence:

\begin{itemize}
\item $\mathbb{F}\Vdash\theta\leq p\ \&\ \eta\leq\iota\ \Rightarrow\ \phi\leq\psi$;
\item $\mathbb{F}\Vdash\eta(\theta/p)\leq\iota(\theta/p)\ \Rightarrow\ \phi(\theta/p)\leq\psi(\theta/p)$.
\end{itemize}

$\Downarrow$: Assume that $\mathbb{F}\Vdash\theta\leq p\ \&\ \eta\leq\iota\ \Rightarrow\ \phi\leq\psi.$ Then for any valuation $V$ on $\mathbb{F}$, if $\mathbb{F},V\Vdash\eta(\theta/p)\leq\iota(\theta/p),$ then take $V'=V^{p}_{V(\theta)}$, then since $p$ does not occur in $\theta$, we have $V'(\theta)=V(\theta)=V'(p),$ therefore $V(\eta(\theta/p))=V'(\eta(\theta/p))=V'(\eta),$ similarly $V(\iota(\theta/p))=V'(\iota),$ so from $\mathbb{F},V\Vdash\eta(\theta/p)\leq\iota(\theta/p)$ we get $V'(\eta)\subseteq V'(\iota).$ Therefore $\mathbb{F},V'\Vdash\phi\leq\psi.$ Therefore $V'(\phi)\subseteq V'(\psi)$. Similar to $\eta$ and $\iota$ we get $V'(\phi)=V(\phi(\theta/p))$ and $V'(\psi)=V(\psi(\theta/p)),$ so $\mathbb{F},V\Vdash\phi(\theta/p)\leq\psi(\theta/p).$ By the arbitrariness of $V$ we get $\mathbb{F}\Vdash\eta(\theta/p)\leq\iota(\theta/p)\ \Rightarrow\ \phi(\theta/p)\leq\psi(\theta/p).$\\

$\Uparrow$: Assume $\mathbb{F}\Vdash\eta(\theta/p)\leq\iota(\theta/p)\ \Rightarrow\ \phi(\theta/p)\leq\psi(\theta/p).$ Then for any valuation $V$ on $\mathbb{F}$, if $\mathbb{F},V\Vdash\theta\leq p\ \&\ \eta\leq\iota,$ then $V(\theta)\subseteq V(p),V(\eta)\subseteq V(\iota).$ Therefore by the polarity of $p$ in $\eta(\theta/p)$ and $\iota(\theta/p)$ we have that $V(\eta(\theta/p))\subseteq V(\eta)\subseteq V(\iota)\subseteq V(\iota(\theta/p)).$ So from $\mathbb{F}\Vdash\eta(\theta/p)\leq\iota(\theta/p)\ \Rightarrow\ \phi(\theta/p)\leq\psi(\theta/p)$ we get $V(\phi)\subseteq V(\phi(\theta/p))\subseteq V(\psi(\theta/p))\subseteq V(\psi),$ so $\mathbb{F},V\Vdash\phi\leq\psi$. Therefore by the arbitrariness of $V$ we get $\mathbb{F}\Vdash\theta\leq p\ \&\ \eta\leq\iota\ \Rightarrow\ \phi\leq\psi.$
\end{proof}

\section{Conclusions}\label{aSec:Discussion}

In this paper, we study the correspondence theory of intuitionistic modal logic in modal Fairtlough-Mendler semantics, which is the intuitionistic modal counterpart of possibility semantics. Our study can be regarded as the study of correspondence theory for complete Heyting algebras with complete operators which are not necessarily perfect. 

We proposed the general principles to choose the interpretations of the expanded modal language used in the algorithm $\mathsf{ALBA}$, and apply it in the current setting. Therefore, we interpret the nominals in our setting as the refined regular open closures of singletons, which is join-dense in $\mathbb{H}_{\mathsf{RO}_{12}}$ and can be expressed in a first-order way.

For future directions, we mention the following:

\begin{itemize}
\item In the present paper, we use a join-dense set of $\mathbb{H}_{\mathsf{RO}_{12}}$ to interpret the nominals. Whether a meet-dense set of $\mathbb{H}_{\mathsf{RO}_{12}}$ which can be expressed in a first-order way can be found is a question related to whether we can expand the fragment of inductive formulas here, which is also related to finding correspondents of the $\mathsf{KC}$ formula $\neg p\lor\neg\neg p$ and the $\mathsf{LC}$ formula $(p\to q)\lor(q\to p)$ in the present semantic setting.
\item In the present paper, we only have $\Box$ as the modality, since it interacts well with arbitrary intersection. A future direction is to find a way to incorporate $\Diamond$ into the picture.
\item In the present paper, we consider complete Heyting algebra expansions, therefore it is natural to also consider the correspondence theory of complete lattice expansions which are not necessarily perfect.
\item In the present paper, we do not consider the canonicity and completeness concepts related to the intuitionistic modal version of possibility semantics, which is a future direction.
\end{itemize}

\paragraph{Acknowledgement} The research of the author is supported by the Taishan Young Scholars Program of the Government of Shandong Province, China (No.tsqn201909151).

\bibliographystyle{abbrv}
\bibliography{Bi_Relational}

\end{document}